\documentclass[11pt,reqno]{amsart}

\usepackage{floatrow}

\usepackage{amsmath}
\usepackage{amssymb}
\usepackage[left=1in,top=1in,right=1in,bottom=1in]{geometry}
\usepackage{epsfig}
\usepackage[normalem]{ulem}
\usepackage[usenames,dvipsnames]{color}
\usepackage{todonotes}

\usepackage[colorlinks, citecolor=black, linkcolor=black, urlcolor=black]{hyperref}
\usepackage{bm}
\usepackage{algorithm}
\usepackage{algpseudocode}

\usepackage{enumitem}

\allowdisplaybreaks

\usepackage{comment} 
\usepackage{hyperref}
\usepackage{cleveref}
\crefname{subsection}{subsection}{subsections}

\usepackage{tikz}
\usetikzlibrary{backgrounds}
\usetikzlibrary{intersections}

\newtheorem{theorem}{Theorem}[section]
\newtheorem{lemma}[theorem]{Lemma}

\newtheorem{claim}[theorem]{Claim}

\newcommand{\E}{\mathbb E}

\newcommand{\Var}{\mathbb V\textrm{ar}}
\newcommand{\Prob}{\mathbb{P}}
\newcommand{\Bin}{\mathrm{Bin}}

\newcommand{\Nn}{{\mathbb N}}

\newcommand{\scr}{\mathcal}
\newcommand{\mb}{\mathbb}

\newcommand{\whp}{{\it w.h.p.}}

\theoremstyle{definition}
\newtheorem{remark}{Remark}

\title[Triangle-Saturated Graphs in the Semi-Random Graph Process]{Triangle-Saturated Graphs \\ in \\ the Semi-Random Graph Process}

\author{Felix Christian Clemen}
\address{Department of Mathematics and Statistics, University of Victoria, Victoria, Canada.}
\email{fclemen@uvic.ca}

\author{Pawe\l{} Pra\l{}at}
\address{Department of Mathematics, Toronto Metropolitan University, Toronto, Canada}
\email{pralat@torontomu.ca}

\date{}

\begin{document}

\begin{abstract}
The semi-random graph process is an adaptive random graph process in which an online algorithm is initially given an empty graph on $n$ vertices. In each round, a vertex $u$ is presented to the algorithm independently and uniformly at random. The algorithm then adaptively selects a vertex $v$, and adds the edge $uv$ to the graph. We also consider the offline version of the process in which the algorithm is given the entire sequence of random vertex choices before the selection takes place. For a given graph property, the objective of the algorithm is to force the graph to satisfy this property asymptotically almost surely in as few rounds as possible. 

In this paper, we focus on the property of being triangle-saturated and establish upper and lower bounds on the number of rounds required to construct a triangle-saturated graph in both the online and offline versions of the process.
\end{abstract}

\maketitle

\section{Introduction and Main Results}

The \textit{semi-random graph process} was suggested by Peleg Michaeli, formally introduced in 2020~\cite{ben2020semi}, and studied recently~\cite{behague2022,behague2026creating,ben-eliezer_fast_2020,burova2025semi,gao2022perfect,koerts2022k,macrury2022sharp,molloy2025perfect}, especially in the context of Hamiltonian cycles~\cite{frieze2025building,frieze2022hamilton,gao2020hamilton,gao2022fully,molloy2025perfect}. It is an example of an \textit{adaptive} random graph process, in that an \textit{algorithm} has partial control over which random edges are added in each step. Specifically, the algorithm begins with the empty multigraph $H_0$ on vertex set $[n] = \{1, \ldots, n\}$, and in each \textit{step} (or round) $t \in \mb{N}$, a vertex $u_t$ is chosen independently and uniformly at random (u.a.r.) from $[n]$. The algorithm is \textit{online}, in that it is given the graph $H_{t-1}$ and vertex $u_t$ and must select a vertex $v_t \in [n]$ and add the edge $u_tv_t$ to $H_{t-1}$ to form $H_{t}$. Thus, it decides on $v_t$ without full knowledge of which vertices will be randomly drawn in the future. 

In this paper, the goal of the online algorithm is to make the underlying simple graph $G_t$ satisfy a given graph property $\scr{P}$ as quickly as possible. Here, $G_t$ is obtained from $H_t$ by deleting loops and replacing each collection of parallel edges by a single edge. If the online algorithm chooses $v_t$ u.a.r.\ from $[n]\setminus\{u_t\}$ in every round, then each edge of $K_n$ is selected independently and uniformly at random. Thus, this strategy produces the Erdős--Rényi random graph process with replacement. The main focus in the literature is to understand how intelligent decision-making can accelerate the appearance of graph properties $\scr{P}$.

There are some interesting variants of the semi-random process, including the \emph{offline} version that was introduced in~\cite{ben2020semi} together with the original model. In this variant of the process, the algorithm is given the entire sequence of random vertex choices before the selection takes place. Clearly, this variant is more powerful and typically one can achieve property $\scr{P}$ faster compared to the original model. However, there are properties for which this additional power might not accelerate the process. For example, both online and offline models need $n-1$ rounds to create a connected graph. 

\medskip

While the semi-random graph process has been studied extensively in recent years, discrete processes in which an algorithm has partial control over its random steps have been studied previously. One of the first such examples is the work of Azar et al.~\cite{azar_balanced} in the context of the balls-into-bins problem. By allowing an algorithm a small amount of \textit{adaptivity}, Azar et al.~proved that the maximum load on any bin can be reduced by an exponential factor in comparison to a purely random (and non-adaptive) strategy. This phenomenon has since been described as the ``power of two choices''. 

The \textit{Achlioptas process} is another example of an adaptive random process; it was proposed by Dimitris Achlioptas and first formally studied in~\cite{bohman2001avoiding}. The Achlioptas process, too, begins with the empty graph $G_0$ on vertex set $[n]$. In each round $t \in \mb{N}$, the online algorithm is presented two distinct edges $e^{1}_t, e^{2}_t$ drawn u.a.r.\ from the edges on vertex set $[n]$ that were \textit{not} previously chosen (i.e., not in $G_{t-1}$).  The online algorithm then chooses precisely one of $e^{1}_t, e^{2}_t$, and then adds it to $G_{t-1}$, yielding the graph $G_{t}$. In contrast to the semi-random process, the objective first considered in~\cite{bohman2001avoiding} is to \textit{delay} the construction of a graph satisfying a property $\scr{P}$ for as many rounds as possible. Achlioptas asked what can be done if $\scr{P}$ corresponds to the existence of a giant component (i.e., a connected component of size $\Omega(n)$). Bohman and Frieze analyzed a greedy strategy which does \textit{not} build a giant component for $0.535 n$ rounds, which is strictly larger than the threshold at which a giant component appears in the Erdős–Rényi random graph process~\cite{erdHos1960evolution}. Improvements have since been made on determining the optimal algorithm for delaying the appearance of a giant component. The best known lower bound of $0.829n$ is due to Spencer and Wormald~\cite{spencer_birth}, and the best known upper bound of $0.944n$ is due to Cobârzan~\cite{cob_2014}.  Numerous works have studied other properties and extensions of the Achlioptas process. Most related to our work is~\cite{krivelevich2010hamiltonicity}, whose goal is to force $G_t$ to be Hamiltonian as \textit{quickly} as possible. We refer the reader to the introduction of~\cite{Fraiman2022OnTP} for an in-depth overview of the literature on the Achlioptas process.

\subsection{Main Result: Creating Triangle-saturated Graphs} 

In this paper, we focus on the property of being triangle-saturated. A \emph{triangle-saturated} (or $K_3$-saturated) graph is a graph that is triangle-free but adding any missing edge creates at least one new triangle. This class of graphs is also widely known as maximally triangle-free. 

Classical extremal results for triangle-saturated graphs provide useful benchmarks for determining how quickly such graphs can be constructed in the semi-random process. The maximum number of edges in a triangle-saturated graph on $n$ vertices is $\lfloor n^2/4\rfloor$, and the unique extremal graph is the complete bipartite graph $K_{\lfloor n/2\rfloor,\lceil n/2\rceil}$ by Mantel's Theorem~\cite{mantel1907vraagstuk}. On the other hand, the minimum number of edges in a triangle-saturated graph on $n$ vertices is $n-1$, which immediately gives a lower bound of $n-1$ rounds for the semi-random process. The unique graph attaining this minimum is the star $K_{1,n-1}$ by the Erdős--Hajnal--Moon Theorem~\cite{erdos1964problem}.

\medskip

To create a triangle-saturated graph in the semi-random process, it is tempting to simulate the famous triangle-free process that was introduced by Bollob\'as and Erd\H{o}s (see~\cite{bollobas2009random}). The process starts with an empty graph on $n$ vertices and adds edges one-by-one, chosen uniformly at random, provided that the newly added edge does not form a triangle. The process terminates when no more edges can be added (resulting in a random triangle-saturated graph). Bohman and Keevash proved that \whp\footnote{A property is said to hold \emph{with high probability} (\whp) if it holds with probability tending to one as $n \to \infty$}\ the triangle-free process terminates after $(1/(2\sqrt{2})+o(1)) (\log n)^{1/2} n^{3/2}$ rounds~\cite{bohman2021dynamic}. 

To simulate the triangle-free process, in each round the online semi-random process may choose a vertex $r$ u.a.r.\ from $[n]$ and let $v_t=r$ if $r\neq u_t$, $u_tr\notin E(G_{t-1})$, and adding $u_tr$ does not create a triangle; otherwise, let $v_t=u_t$ (that is, create a loop). Clearly, this coupling ensures that after $t$ rounds, the graph obtained by the semi-random process after deleting the loops is exactly the graph obtained by the triangle-free process after $t-s$ rounds, where $s$ denotes the number of loops created during the first $t$ rounds. Initially, very few loops are created but later on the semi-random process will be mostly creating loops. Hence, this approach will not give any useful upper bound. However, not surprisingly, a fully adaptive algorithm can do much better than that anyway. 

\medskip

Let $w_{\mathrm{on}}(n,H)$ denote the smallest integer $m$ for which there exists an online strategy that constructs an $H$-saturated graph after $m$ rounds with probability at least $1/2$. Define $w_{\mathrm{off}}(n,H)$ analogously, with the strategy allowed to know the entire sequence of presented vertices in advance.

In this paper, we prove the following results that are asymptotic by nature. For example, to prove an upper bound for $w_{\mathrm{on}}(n,H)$ or $w_{\mathrm{off}}(n,H)$, we fix a strategy and a function $u = u(n)$ and show that \whp\ this specific strategy creates an $H$-saturated graph after $u$ rounds. Similarly, to prove a lower bound, we fix a function $\ell = \ell(n)$ and show that \whp\ no strategy creates an $H$-saturated graph after $\ell$ rounds. As a result, the bounds claimed below hold assuming that $n$ is large enough. 

For the offline variant, we prove the following bounds. 

\begin{theorem}
\label{thm:off_upper}
\begin{eqnarray*}
w_{\mathrm{off}}(n,K_3)
&\le&
n \log n - n \log\log n + n \\
&=& n \log n - (1+o(1)) n \log\log n.
\end{eqnarray*}
\end{theorem}

\begin{theorem}
\label{thm:off_lower}
\begin{eqnarray*}
w_{\mathrm{off}}(n,K_3)
&\ge& 
n \log n - 2 n \log\log n - n \log 7 \\
&=& n \log n - (2+o(1)) n \log\log n.
\end{eqnarray*}
\end{theorem}

For the original (online) process, we prove the following bounds. 

\begin{theorem}
\label{thm:on_upper}
Let $\omega=\omega(n)$ be any function tending to infinity as $n \to \infty$. Then,
$$
w_{\mathrm{on}}(n,K_3) 
\le 
n \log n + n \omega.
$$
\end{theorem}

\begin{theorem}
\label{thm:on_lower}
\begin{eqnarray*}
w_{\mathrm{on}}(n,K_3) 
&\ge& 
n \log n - n \log \log n - 3 n \log \log \log n \\ 
&=& n \log n - (1+o(1)) n \log\log n.
\end{eqnarray*}
\end{theorem}

Choosing any function $\omega\to\infty$ such that $\omega=o(\log\log n)$ in Theorem~\ref{thm:on_upper}, we see that in both the online and offline settings the difference between our upper and lower bounds is $(1+o(1))n\log\log n$. The proofs of the upper bounds reveal a natural bottleneck, suggesting that the upper bounds are closer to the truth, although we were unable to close the remaining gap.

Let us also mention that one should be able to generalize the arguments and extend our results to $K_k$-saturated graphs for some fixed integer $k \ge 3$. For simplicity, we concentrate on triangle-saturated graphs. 

Proofs of Theorems~\ref{thm:off_upper} and~\ref{thm:off_lower} can be found in Section~\ref{sec:offline} whereas 
proofs of Theorems~\ref{thm:on_upper} and~\ref{thm:on_lower} are provided in Section~\ref{sec:online}. 

\subsection{Further Related Work} 

The seminal paper~\cite{ben2020semi} showed that the semi-random graph process is general enough to simulate several well-studied random graph models by using appropriate strategies. In the same paper, the process was studied for various natural properties such as having minimum degree $k \in \Nn$ or having a fixed graph $H$ as a subgraph. In particular, it was shown that \whp\ one can construct $H$ in less than $n^{(d-1)/d} \omega$ rounds where $d \ge 2$ is the degeneracy of $H$ and $\omega = \omega(n)$ is any function that tends to infinity as $n \to \infty$. This property was recently revisited in~\cite{behague2022}, where a conjecture from~\cite{ben2020semi} was proven for any graph $H$: \whp\ it takes at least $n^{(d-1)/d} / \omega$ rounds to create $H$. The property of having cliques of order tending to infinity as $n\to \infty$ was investigated in~\cite{gamarnik2023cliques}. In \cite{koerts2022k}, $k$-factors and $k$-connectivity were studied. 

\smallskip

Another property studied in the context of semi-random processes is that of having a perfect matching. Since the $2$-out process has a perfect matching \whp~\cite{WALKUP1980}, and the semi-random process can simulate the 2-out process, we immediately get an upper bound of $(2+o(1))n$. By simulating the semi-random process with another random graph process known to have a perfect matching \whp~\cite{pittel}, the bound can be improved to $(1+2/e+o(1))n < 1.73576n$~\cite{ben2020semi}. This bound was improved by investigating another fully adaptive algorithm~\cite{gao2022perfect}, giving the current best bound of $1.20524n$. The same paper improves the lower bound observed in~\cite{ben2020semi} of $(\ln(2)+o(1))n > 0.69314n$ to $0.93261n$. 

\smallskip

Next, we discuss results for the property of having a Hamiltonian cycle. It was shown in~\cite{bohman2009hamilton} that the $3$-out process generates a random graph that is Hamiltonian \whp\ Since the semi-random process can simulate the 3-out process, we get an upper bound of $(3+o(1))n$. A new upper bound was obtained in~\cite{gao2020hamilton} in terms of an optimal solution to an optimization problem whose value is believed to be at most $2.61135n$ by numerical support. The upper bound $(3+o(1))n$ obtained by simulating the $3$-out process is \textit{non-adaptive}. That is, the strategy does not depend on the history of the semi-random process. The improvement in~\cite{gao2020hamilton} is adaptive but in a weak sense. The strategy consists of four phases, each lasting a linear number of rounds, and the strategy is adjusted only at the end of each phase: for example, the algorithm might identify vertices of low degree, and then focus on them during the next phase. In~\cite{gao2022fully}, a fully adaptive strategy was proposed: at every step $t$, it pays attention to $G_{t-1}$ and $u_t$. As expected, such a strategy creates a Hamiltonian cycle substantially faster, and it improves the upper bound from $2.61135n$ to $2.01678n$. A further improvement in~\cite{frieze2022hamilton} brings the upper bound down to $1.84887n$. All ideas are combined together in~\cite{frieze2025building} to reduce it further, to $1.81701n$.

For the lower bound, it was observed in~\cite{ben2020semi} that because any Hamiltonian graph has minimum degree at least $2$, the lower bound is at least $\bigl(\ln 2+\ln(1+\ln 2)+o(1)\bigr)n \ge 1.21973n$. This was subsequently improved by $10^{-8}n$ in~\cite{gao2020hamilton}. By identifying structures created by the semi-random process whose edges cannot all lie on a Hamiltonian cycle, the bound was further raised to $1.26575n$ in~\cite{frieze2025building,gao2022fully}.

\smallskip

Other adaptive random graph processes and variants of the semi-random graph process have been considered in the literature. The \textit{semi-random tree process} is introduced in~\cite{burova2025semi}, where in each round, a random spanning tree of $K_n$ is presented to the algorithm, who chooses one of the edges to keep.  In~\cite{Harjas}, $k$ random vertices rather than just one are offered, and the algorithm chooses one of them before creating an edge. In~\cite{macrury2022sharp}, a general definition of an \textit{adaptive random graph process} is proposed. By parameterizing it appropriately, one recovers the Achlioptas process, the semi-random graph process, as well as the models of~\cite{burova2025semi} and~\cite{Harjas}. In~\cite{gilboa2021semi}, the vertices offered by the process follow a random permutation. Finally, hypergraphs are investigated in~\cite{behague2022,behague2026creating,molloy2025perfect}.

\section{Notation and Auxiliary Observations}\label{sec:auxiliary}

Throughout the paper, all logarithms are natural. Whenever a non-integer expression is used to specify a number of rounds, we take its floor.

\subsection{Notation}

An unordered pair $ab\notin E(G)$ in a triangle-free graph $G$ is called \emph{open} if adding $ab$ to $G$ does not create a triangle. Note that a triangle-free graph is triangle-saturated if and only if it does not contain an open pair. An unordered pair $ab\notin E(G)$ in a triangle-free graph $G$ is called \emph{closed} if adding $ab$ to $G$ creates a triangle.

For a graph $G$ and a vertex $x\in V(G)$, we write $N_G(x):=\{y\in V(G):xy\in E(G)\}$ for the \emph{neighbourhood} of $x$ in $G$, and $\deg_G(x):=|N_G(x)|$ for its \emph{degree}.

\subsection{Coupon Collector Estimates}

The number of vertices presented by the semi-random process is well-understood as it is related to the famous coupon collector problem. We record a few properties that we will use in our proofs.

\begin{lemma}\label{lem:coupon_collector}
Fix any function $T=T(n)$.
For any vertex $i \in [n]$, let $X_i$ be the random variable counting how many times vertex $i$ is presented by the semi-random process during the first $T$ steps.
Let $Z$ be the random variable counting how many vertices were not presented by the semi-random process during the first $T$ steps. 
Then, the following properties hold.
\begin{itemize}
\item [(a)] Suppose that $T/n - \log n \to - \infty$ as $n \to \infty$. Then, \whp
$$
Z = (1+o(1)) n \exp \left( - \frac{T}{n} \right) \to \infty.
$$ 
\item [(b)] Suppose that $T/n - \log n \to \infty$ as $n \to \infty$. Then, \whp\ $Z = 0$.
\item [(c)] Suppose that $T = (1+o(1)) n \log n$. Then, \whp\ no vertex is presented more than $3 \log n$ times, that is,
$$
\max_{i \in [n]} X_i \le 3 \log n.
$$
\end{itemize}
\end{lemma}

\begin{proof}
We start with the proof of part (a). Suppose that $T = T(n) = n(\log n - \omega)$, where $\omega = \omega(n) \to \infty$ as $n \to \infty$. Clearly, for any $i \in [n]$, the probability that vertex $i$ is \emph{not} presented by the semi-random process during the first $T$ steps (that is, that $u_t \neq i$ for all $t \in [T]$) can be estimated as follows:
$$
\Prob ( X_i = 0 ) ~=~ \left( 1 - \frac 1n \right)^T ~=~ \exp \left( - \frac{T}{n} + O \left( \frac {T}{n^2} \right) \right) ~=~ (1+o(1)) \exp \left( - \frac{T}{n} \right),
$$
and so $\E [Z] = (1+o(1)) n \exp( - T/n) = (1+o(1)) e^{\omega} \to \infty$ as $n \to \infty$. Similarly, 
$$
\E [Z(Z-1)] ~=~ n (n-1) \left( 1 - \frac 2n \right)^T ~=~ (1+o(1)) n^2 \exp \left( - \frac{2T}{n} \right) ~=~ (1+o(1)) ( \E [Z] )^2,
$$
and so $\Var [Z] = o ( \E [Z] )^2$. It follows from Chebyshev's inequality that \whp\ 
$$
Z = (1+o(1)) \E [Z] = (1+o(1)) n \exp( - T/n). 
$$
This finishes the proof of part~(a).

\medskip

For part~(b), suppose that $T = T(n) = n(\log n + \omega)$, where $\omega = \omega(n) \to \infty$ as $n \to \infty$. Arguing as before and using $1-x\leq e^{-x}$, we obtain \[ \E[Z] = n\left(1-\frac1n\right)^T \leq n\exp\left(-\frac{T}{n}\right) = \exp\left(\log n-\frac{T}{n}\right) \to 0. \] Hence, by Markov's inequality, we get that \whp\ $Z = 0$. This finishes the proof of part~(b).

\medskip

Finally, for part~(c), suppose that $T = (1+o(1)) n \log n$ and fix any vertex $i \in [n]$. Clearly, $\E [X_i] = T/n = (1+o(1)) \log n$. It follows from Chernoff's bound that 
\begin{align*}
\Prob ( X_i \ge 3 \log n ) &= \Prob \Big( X_i \ge \E[X_i] + (2+o(1)) \log n \Big) \\
&\le \exp \left( - \left( \frac {2^2}{2(1+2/3)} + o(1) \right) \log n \right) = o (1/n). 
\end{align*}
Hence, by the union bound over $n$ vertices, \whp\ $X_i \le 3 \log n$ for all $i \in [n]$. This finishes the proof of part~(c). 
\end{proof}

\section{Offline Variant -- Proof of Theorems~\ref{thm:off_upper} and~\ref{thm:off_lower}}\label{sec:offline}

Let us start with the upper bound as it is simpler.

\begin{proof}[Proof of Theorem~\ref{thm:off_upper}]
Fix $T = n(\log n - \log \log n + 1) = (1+o(1)) n \log n$. Let $Z_T$ be the number of vertices not presented during the first $T$ steps. It follows from Lemma~\ref{lem:coupon_collector}(a) that \whp\ $Z_T = (1+o(1)) n \exp(-T/n) = (1+o(1)) e^{-1} \log n$.

Now, let $Y_T$ be the random variable counting how many times vertex $1$ is presented by the semi-random process during the first $T$ steps. Clearly, $Y_T \sim \Bin(T,1/n)$ with $\E[Y_T] = T/n = (1+o(1)) \log n$. It follows immediately from Chernoff's inequality that \whp\ $Y_T = (1+o(1)) \E[Y_T] = (1+o(1)) \log n$.

\medskip

Our strategy is simple. We build a star centred at vertex $1$, which is a triangle-saturated graph. Each time vertex $1$ is presented by the semi-random process, we connect it to one of the vertices that are not presented during the first $T$ steps. On the other hand, each time a vertex other than $1$ is presented, we connect it with~$1$. Since \whp\ vertex $1$ is presented more often than the number of vertices not presented at all (\whp\ $(1+o(1)) \log n = Y_T > Z_T = (1+o(1)) e^{-1} \log n$), the strategy works \whp\ This completes the proof of the theorem.
\end{proof}
\begin{remark}
Note that in the proof of the above theorem we could have selected a vertex presented the most often by the semi-random process instead of vertex $1$. This vertex could take care of slightly more vertices that were not presented at all which, in turn, would improve our upper bound. However, \whp\ each vertex $i \in [n]$ is presented $O(\log n)$ times (see Lemma~\ref{lem:coupon_collector}(c)) so this would only improve the bound by an additive term $O(n)$. Hence, we opted for a slightly easier proof.
\end{remark}
\medskip

We now move to the lower bound.

\begin{proof}[Proof of Theorem~\ref{thm:off_lower}]
Fix $T = n(\log n -2 \log \log n-\log 7) = (1+o(1)) n \log n$. Let $G_T$ be the graph constructed by an arbitrary offline strategy. If \(G_T\) contains a triangle, then it is not triangle-saturated, and
there is nothing to prove. We may therefore assume that \(G_T\) is
triangle-free.

As in the proof of Theorem~\ref{thm:off_upper}, let $Z_T$ be the number of vertices \emph{not} presented by the semi-random process during the first $T$ steps. It follows from Lemma~\ref{lem:coupon_collector}(a) that \whp\ $Z_T = (7+o(1)) \log^2 n$. Let $X_i$ be the random variable counting how many times vertex $i \in [n]$ is presented by the semi-random process during the first $T$ steps. It follows from Lemma~\ref{lem:coupon_collector}(c) that \whp\ $\max_{i \in [n]} X_i \le 3 \log n$. Since we aim for a statement that holds \whp, we may assume that these two properties hold. 

Let $I \subseteq [n]$ be the set of vertices not presented by the semi-random process. Clearly, $I$ induces an independent set in $G_T$ and, by our assumption, $|I|= Z_T=(7+o(1)) \log^2 n$. Our goal is to show that there exists an open pair $ab$, where $a\in I^c$ and $b\in I$. This will complete the proof of the theorem, since such a pair witnesses that $G_T$ is not triangle-saturated. 

Suppose that a cross-pair $ab$, with $a\in I^c$ and $b\in I$, is closed. Then, there exists a vertex $w\in [n]$ such that $wa,wb\in E(G_T)$. Since $I$ is an independent set, $w\in I^c$. Therefore, the number of closed cross-pairs $ab$ is at most
\begin{align}
\label{countclosed}
\sum_{w\in I^c}|N_{G_T}(w)\cap I^c| \cdot|N_{G_T}(w)\cap I|.
\end{align}
Note that a closed pair may be counted multiple times (if multiple common neighbours exist) in \eqref{countclosed}, but this only strengthens the upper bound. Every edge between \(w\) and \(I\) must be created in a round in which \(w\) is presented. Hence, for every \(w\in I^c\),
\[
|N_{G_T}(w)\cap I|
\leq X_w
\le \max_{i \in [n]} X_i 
\leq 3\log n,
\] 
where the last inequality follows from our assumption. 
Consequently, the number of closed cross-pairs $ab$ is at most
\begin{align*}
\sum_{w\in I^c}|N_{G_T}(w)\cap I^c| \cdot|N_{G_T}(w)\cap I|&\leq 3\log n \cdot\sum_{w\in I^c}|N_{G_T}(w)\cap I^c| \leq  3\log n\cdot\sum_{w\in I^c}\deg_{G_T}(w) \\ 
&\leq 3\log n \cdot 2T \leq (6+o(1)) n \log^2 n.
\end{align*}
Moreover, trivially, the number of pairs $ab$ such that $ab\in E(G_T)$ is at most $T$. Hence the number of cross-pairs that are either edges or closed is at most 
$$T+(6+o(1))n \log^2 n= (6+o(1)) n \log^2 n< (7+o(1))n\log ^2n =|I||I^c|.$$ 
It follows that some pair between $I$ and $I^c$ is open (neither an edge nor closed). Hence, $G_T$ is not triangle-saturated. This completes the proof of the theorem. 
\end{proof}

\section{Online Variant -- Proof of Theorems~\ref{thm:on_upper} and~\ref{thm:on_lower}}\label{sec:online}

As in the previous section, we start with the upper bound.

\begin{proof}[Proof of Theorem~\ref{thm:on_upper}]
Let $\omega=\omega(n)$ be any function tending to infinity as $n \to \infty$. Fix $T = n(\log n + \omega)$. It follows from Lemma~\ref{lem:coupon_collector}(b) that \whp\ each vertex is presented at least once. 

\medskip

We can build a star very easily, creating a triangle-saturated graph. We can pick an arbitrary vertex (say, vertex $n$) before the process starts. Then, in each step of the process we connect the presented vertex $u_t$ to this chosen vertex (that is, fix $v_t = n$ for all $t \in [T]$). The desired star is created in at most $T$ steps \whp, which finishes the proof of the theorem. 
\end{proof}

\begin{remark}
Note that in the proof of the above theorem, one could consider the following, slightly better strategy. One may still connect vertices presented by the semi-random process to vertex $n$, unless $n$ is presented. If $n$ is presented, then one may connect it to an arbitrary vertex that is not yet connected to $n$. This would speed up the creation of a star but not by much. The original strategy is expected to finish in 
$
\sum_{i=1}^{n-1} \frac {n}{i} = n H_n +O(1) = n \log n + \gamma n + O(1)
$ 
steps, where $H_n$ is the harmonic number and $\gamma \approx  0.57721$ is the Euler-Mascheroni constant. The better strategy is expected to finish in 
$
\sum_{i=1}^{n-1} \frac {n}{i+1} = n H_n - n = n \log n - (1-\gamma) n + O(1)
$ 
steps, improving the expectation only by a linear term. 
\end{remark}

\medskip

We now move to the lower bound which is much more involved. 

\begin{proof}[Proof of Theorem~\ref{thm:on_lower}]
 We will analyze the process in two phases. Let us fix
$$
T_1 = n(\log n -3 \log \log n) 
\quad \text{and} \quad 
T=n(\log n - \log \log n-3 \log \log \log n).
$$ 
Fix an arbitrary online strategy, and let \(G_t\) denote the graph
constructed after \(t\) rounds. If \(G_T\) contains a
triangle, then it is not triangle-saturated, and there is nothing to
prove. We may therefore assume that \(G_T\) is triangle-free.

Let $Z_{T_1}$ be the number of vertices \emph{not} presented by the semi-random process during the first $T_1$ steps. It follows from Lemma~\ref{lem:coupon_collector}(a) that \whp\ $Z_{T_1} = (1+o(1)) \log^3 n$.
For \(s\in[T]\) and \(v\in[n]\), let
\[
M_s(v):=\bigl|\{t\in[s]:u_t=v\}\bigr|
\]
be the random variable counting how many times vertex \(v\) is presented by the semi-random process during the first \(s\) rounds. It follows from Lemma~\ref{lem:coupon_collector}(c) that \whp\ $\max_{v \in [n]} M_T(v) \le 3 \log n$.

Let $I:=\{v\in [n]: M_{T_1}(v)=0\}$ be the set of vertices not presented by the semi-random process during the first $T_1$ steps. Since \whp\ $|I| = Z_{T_1} = (1+o(1)) \log^3 n$, \whp\ there are $(1+o(1)) n \log^3 n$ cross pairs $ab$ with $a\in I$ and $b\in I^c$. Arguing exactly as in the proof of Theorem~\ref{thm:off_lower}, regardless of how graph $G_{T_1}$ is constructed, we conclude that \whp\ almost all of these pairs are still open. 
\begin{claim}
\label{crosspairIIc}
The following holds \whp: the number of cross pairs $ab$, $a\in I, b\in I^c$ that are edges or closed at time $T_1$ is $O(n \log^2 n)$. 
\end{claim}
\begin{proof}
The proof is identical to the corresponding argument in the proof of
Theorem~\ref{thm:off_lower}, with \(T\) replaced by \(T_1\).
\end{proof}

The above claim shows that on average a vertex from $I$ is part of $O(n/\log n)$ cross pairs that are edges or closed. Clearly, there could be some vertices in $I$ that are part of, say, more than $n / \log \log n$ such cross pairs. However, a simple averaging argument shows that there cannot be too many such vertices. Formally, for each \(a\in I\), let
\[
c(a):=
\bigl|\{b\in I^c:ab\text{ is an edge or is closed in }G_{T_1}\}\bigr|,
\]
and define
\[
I':=
\left\{
a\in I:
c(a)\leq \frac{n}{\log\log n}
\right\}.
\]

\begin{claim}
\label{claim:Iprime}
The following holds \whp: 
$|I'|=(1+o(1))|I|$.
\end{claim}
\begin{proof}
By Claim~\ref{crosspairIIc}, \whp\
$
\sum_{a\in I}c(a)=O(n\log^2 n).
$
On the other hand, by the definition of $I'$,
\[
\sum_{a\in I\setminus I'}c(a)
>
|I\setminus I'|\frac{n}{\log\log n}.
\]
Consequently, \whp\ $|I\setminus I'|
=
O\bigl( (\log n)^2 (\log\log n) \bigr)
=
o(\log^3 n)$. Since $|I|=(1+o(1))\log^3 n$ \whp, we conclude that \whp\ $|I'|=(1+o(1))|I|$.
\end{proof}

Let $J\subseteq [n]$ be the set of vertices not presented by the semi-random process by time $T$, that is, $J:=\{v\in [n]: M_T(v)=0\}$. It follows from  Lemma~\ref{lem:coupon_collector}(a) that \whp\ 
$$
|J| = Z_T = (1+o(1)) (\log n) (\log \log n)^3.
$$ 
Note that $J\subseteq I$ because $T\geq T_1$. More importantly, $J$ is a random subset of $I$ of cardinality $Z_T$. Consider $J'=J\cap I'$. By Claim~\ref{claim:Iprime}, almost all vertices in $I$ are in $I'$.

Conditioned on the history up to time \(T_1\) and on \(|J|\), the set
\(J\) is uniformly distributed among all \(|J|\)-subsets of \(I\).
Consequently,
\[
\mathbb{E}\bigl[|J\setminus I'|
  \,\big|\, I,I',|J|\bigr]
=
|J|\frac{|I\setminus I'|}{|I|}
=
o(|J|),
\]
where the last equality follows from Claim~\ref{claim:Iprime}.
Thus, by Markov's inequality, $|J\setminus I'|=o(|J|)$
\whp\ Therefore,
\[
|J'|=(1+o(1))|J|
=(1+o(1))(\log n)(\log\log n)^3.
\]

Next, we bound the number of 
cross pairs between $J'$ and $J'^c$ that are edges or closed.
\begin{claim}
\label{claim:initial-bad-pairs}
The following holds \whp: at time $T_1$, the number of pairs between $J'$ and $J'^c$ that are edges or closed is $o(n|J'|)$.
\end{claim}

\begin{proof}
Since $J'\subseteq I'$, there are at most
$$
\sum_{a\in J'}c(a)\leq \frac{n|J'|}{\log\log n}=o(n|J'|)
$$
such pairs between $J'$ and $I^c$. The remaining pairs lie between $J'$ and $I\setminus J'$, of which there are at most
$|J'||I|=O(|J'|\log^3 n)=o(n|J'|)$.
\end{proof}

We now upper bound the number of additional cross pairs between $J'$ and $J'^c$ that become closed during rounds $T_1+1, T_1+2, \ldots, T$. 
\begin{lemma}
\label{newclosed}
The following holds \whp: 
the number of cross pairs between $J'$ and $J'^c$
that become closed during rounds $T_1+1, T_1+2, \ldots, T$ is $O(n (\log n) (\log \log n)^2)$.
\end{lemma}
\begin{proof}
For a set
$S\subseteq[n]$, write
\[
d_t(x,S):=|N_{G_t}(x)\cap S|
\qquad\text{and}\qquad
d_t(x):=d_t(x,[n])=\deg_{G_t}(x).
\]

Suppose that the edge $u_tv_t$ is added at time $t$. Every pair that
becomes closed at time $t$ is either of the form $xv_t$, where
$x\in N_{G_{t-1}}(u_t)$, or of the form $xu_t$, where
$x\in N_{G_{t-1}}(v_t)$. Since $u_t\notin J'$, the newly closed cross
pairs of the second type are counted by $d_{t-1}(v_t,J')$. For the
first type, if $v_t\in J'$, there are at most $d_{t-1}(u_t)$ such
pairs, whereas if $v_t\notin J'$, there are at most
$d_{t-1}(u_t,J')$ such pairs. Thus, the number of cross pairs between
$J'$ and $J'^c$ that become closed at time $t$ is at most
\[
\mathbf 1_{\{v_t\in J'\}}d_{t-1}(u_t)
+d_{t-1}(u_t,J')
+d_{t-1}(v_t,J').
\]
Consequently, the total number $C_{\mathrm{new}}$ of such newly closed cross pairs satisfies the following:
\begin{align*}
C_{\mathrm{new}}
\leq
\sum_{t=T_1+1}^T
 \mathbf 1_{\{v_t\in J'\}}d_{t-1}(u_t)+
\sum_{t=T_1+1}^T d_{t-1}(u_t,J')
+
\sum_{t=T_1+1}^T d_{t-1}(v_t,J').
\end{align*}

Because no vertex of $J'$ is ever presented by the semi-random process, every edge between a vertex $x$ and $J'$ must be created in a round in which $x$ is
presented. Hence, \whp\ for every $s\leq T$ and for any $x \in [n]$, $d_s(x,J')\leq M_s(x)\leq M_T(x)\leq 3\log n$.
It follows that \whp
\begin{align}
\nonumber
C_{\mathrm{new}}
&\leq
\sum_{t=T_1+1}^T
 \mathbf 1_{\{v_t\in J'\}}d_{t-1}(u_t)+
(T-T_1)6 \log n \\
\label{eq:new-closed2}
&\leq \sum_{t=T_1+1}^T
 \mathbf 1_{\{v_t\in J'\}}d_{t-1}(u_t)+
12n (\log n) (\log \log n). 
\end{align}
It remains to estimate the sum in \eqref{eq:new-closed2}. Let
\[
A:=
\sum_{t=T_1+1}^T
 \mathbf 1_{\{v_t\in J'\}}d_{t-1}(u_t)\leq \sum_{t=T_1+1}^T
 d_{t-1}(u_t).
\]
Therefore,
\[
\mathbb E A
\leq
\sum_{t=T_1+1}^T
\mathbb E\bigl[d_{t-1}(u_t)\bigr].
\]
Since $u_t$ is a vertex selected uniformly at random from $[n]$ and independent of $G_{t-1}$,
\[
\mathbb E\bigl[d_{t-1}(u_t)\mid G_{t-1}\bigr]
=
\frac{2e(G_{t-1})}{n}
\leq \frac{2(t-1)}{n}
\leq 2\log n.
\]
Hence,
$$
\mathbb E A
\leq
(2 \log n) (T-T_1)
\leq 4n (\log n) (\log \log n).
$$
By Markov's inequality, \whp\ $A=O(n (\log n) (\log \log n)^2)$. Combining this with \eqref{eq:new-closed2}, we conclude that \whp\
$
C_{\mathrm{new}}=O(n (\log n) (\log \log n)^2). 
$
\end{proof}

With Lemma~\ref{newclosed} at hand, it is very easy to complete the proof of Theorem~\ref{thm:on_lower}. By Lemma~\ref{newclosed} and Claim~\ref{claim:initial-bad-pairs}, at time $T$, \whp\ the total number of cross pairs between $J'$ and $J'^c$ that are closed is $o(n|J'|)+O(n (\log n) (\log \log n)^2)=o(n|J'|)$. The number of edges between $J'$ and $J'^c$ is, trivially, at most $T=O(n \log n)=o(n|J'|)$. On the other hand, \whp\ the total number of cross pairs is equal to $|J'||J'^c| =(1+o(1))n|J'|$.
It follows that \whp\ some cross pair between $J'$ and $J'^c$ is neither an edge
nor closed. Hence $G_T$ is not triangle-saturated \whp, completing the proof. 
\end{proof}

\section{Acknowledgements}
The first author's research is supported by a PIMS Postdoctoral Fellowship (PIMS-20260806-PDF).

\bibliographystyle{plain}

\bibliography{refs.bib}

\end{document}